\documentclass[3p]{elsarticle}

\journal{Elsevier}

\usepackage{amsmath}
\usepackage{amssymb}
\usepackage{upgreek} % needed for \uppi
\usepackage{newpxtext,newpxmath}
\usepackage{mathrsfs} % needed for \mathscr{D}
\usepackage{gnuplot-lua-tikz}
\newcommand{\domD}{\mathscr{D}}
\renewcommand{\pi}{\uppi}
\renewcommand{\Re}{\operatorname{Re}}
\renewcommand{\Im}{\operatorname{Im}}
\DeclareMathOperator{\arcsinh}{\mathrm{arcsinh}}
\DeclareMathOperator{\Ci}{\mathrm{Ci}}
\DeclareMathOperator{\si}{\mathrm{si}}
\DeclareMathOperator{\Order}{\mathrm{O}}
\DeclareMathOperator{\E}{\mathrm{e}}
\DeclareMathOperator{\I}{\mathrm{i}}
\newcommand{\D}{\mathrm{d}}
\newdefinition{definition}{Definition}[section]
\newtheorem{theorem}{Theorem}[section]
\newtheorem{lemma}[theorem]{Lemma}

\newproof{proof}{Proof}

\begin{document}

\begin{frontmatter}

\title{New conformal map for the trapezoidal formula for infinite integrals of
unilateral rapidly decreasing functions\tnoteref{mytitlenote}}
\tnotetext[mytitlenote]{This work was partially supported by JSPS Grant-in-Aid for Young Scientists (B) JP17K14147.}

%% Group authors per affiliation:
%\author{Tomoaki Okayama\fnref{myfootnote}}
\author{Tomoaki Okayama}
\address{Graduate School of Information Sciences, Hiroshima City University,
3-4-1, Ozuka-higashi, Asaminami-ku, Hiroshima 731-3194, Japan}
\ead{okayama@hiroshima-cu.ac.jp}
%\fntext[myfootnote]{Since 1880.}

\author{Tomoki Nomura}
\address{Hitachi Information Engineering, Ltd.,
Hiroshima K building 7F, 6-13, Nakamachi, Naka-ku,
Hiroshima 730-0037, Japan}

\author{Saki Tsuruta}
\address{Hiroshima Municipal Funairi High School,
1-4-4, Funairi Minami, Naka-ku, Hiroshima 730-0847, Japan}

%% or include affiliations in footnotes:
%\author[mymainaddress,mysecondaryaddress]{Elsevier Inc}
%\ead[url]{www.elsevier.com}

%\author[mysecondaryaddress]{Global Customer Service\corref{mycorrespondingauthor}}
%\cortext[mycorrespondingauthor]{Corresponding author}
%\ead{support@elsevier.com}

%\address[mymainaddress]{1600 John F Kennedy Boulevard, Philadelphia}
%\address[mysecondaryaddress]{360 Park Avenue South, New York}

\begin{abstract}
While the trapezoidal formula can attain exponential convergence
when applied to infinite integrals of bilateral rapidly decreasing functions,
it is not capable of this in the case of unilateral rapidly decreasing
functions. To address this issue, Stenger proposed the application of a
conformal map to the integrand such that it transforms into
bilateral rapidly decreasing functions.
Okayama and Hanada modified the conformal map and provided a
rigorous error bound for the modified formula.
This paper proposes a further improved conformal map,
with two rigorous error bounds provided for the improved formula.
%in the same form as in the previous study.
Numerical examples comparing the proposed and existing formulas are
also given.
\end{abstract}


\end{frontmatter}

\section{Introduction and summary}

In this paper,
we are concerned with the trapezoidal formula
for the infinite integral, expressed as
\[
\int_{-\infty}^{\infty} f(x) \D x
\approx
 h\sum_{k=-\infty}^{\infty} f(kh),
\]
where $h$ is a mesh size.
This approximation formula is fairly accurate if the integrand
$f(x)$ is analytic,
which has been known since several decades ago~\cite{Schwaltz,Stenger73}.
For example, the approximation
\[
 \int_{-\infty}^{\infty} \E^{-x^2} \D x
\approx
 h\sum_{k=-\infty}^{\infty} \E^{-(kh)^2}
\]
gives the correct answer in double-precision with $h=1/2$,
and the approximation
\[
 \int_{-\infty}^{\infty} \frac{1}{4+x^2} \D x
\approx
 h\sum_{k=-\infty}^{\infty} \frac{1}{4+(kh)^2}
\]
gives the correct answer in double-precision with $h=1/3$.
In general, however, the infinite sum on the right-hand side
cannot be calculated,
and thus, the sum has to be truncated at some $M$ and $N$ as
\[
\int_{-\infty}^{\infty} f(x) \D x
%\approx
% h\sum_{k=-\infty}^{\infty} f(kh)
\approx
 h\sum_{k=-M}^{N} f(kh).
\]
In the case where $f(x)=\E^{-x^2}$, this approximation requires
$h = 1/2$ and $M=N=12$ to obtain the correct answer in double-precision.
On the other hand, in the case where $f(x)=1/(4+x^2)$,
this approximation requires $h=1/3$ and $M=N=10^{16}$
to obtain the correct answer in double-precision.
This is because $f(x) = \E^{-x^2}$ is a rapidly decreasing function,
i.e., $f$ decays exponentially as $x\to\pm\infty$,
whereas $f(x) = 1/(4+x^2)$ is not.

In the case where the integrand $f(x)$ is not a rapidly decreasing function,
a useful solution is the application of an appropriate conformal map
before applying the (truncated) trapezoidal formula.
When $f(x)$ decays algebraically as $x\to\pm\infty$ like $f(x) = 1 / (4 + x^2)$,
by applying a conformal map $x = \sinh t$, a new integral is obtained:
\[
 \int_{-\infty}^{\infty}f(x)\D x
=\int_{-\infty}^{\infty}f(\sinh t)\cosh t\,\D t,
\]
where the transformed integrand $f(\sinh t)\cosh t$ decays
exponentially as $t\to\pm\infty$.
Therefore, the (truncated) trapezoidal formula should yield
an accurate result when applied to the new integral.
Appropriate conformal maps for certain typical cases
have been usefully summarized by Stenger~\cite{stenger93:_numer,Stenger}.

One of the cases listed in the summary is rather convoluted:
the integrand $f(x)$ decays exponentially as $x\to\infty$,
but decays algebraically as $x\to-\infty$,
like $f(x)=1/\{(4+x^2)(1+\E^{x})\}$.
We refer to such a function as a
\emph{unilateral rapidly decreasing function}.
In such a case,
Stenger~\cite{Stenger} proposed the employment of a conformal map
\begin{align*}
x = \psi(t) = \sinh(\log(\arcsinh(\E^t))),
\end{align*}
and applied the trapezoidal formula as
\begin{equation}
%I%=
\int_{-\infty}^{\infty} f(x) \D x
=\int_{-\infty}^{\infty} f(\psi(t))\psi'(t)\D t
\approx h\sum_{k=-M}^N f(\psi(kh))\psi'(kh).
\label{eq:Stenger-formula}
\end{equation}
Furthermore, by appropriately setting $h$, $M$, and $N$
depending on the given positive integer $n$,
he theoretically analyzed the error as $\Order(\E^{-\sqrt{2\pi d \mu' n}})$,
where $\mu'$ indicates the decay rate of the transformed integrand,
and $d$ indicates the width of the domain in which
the transformed integrand is analytic (described in detail further on).
Okayama and Hanada~\cite{Okayama-Hanada} slightly modified
the conformal map as follows:
\[
 x = \tilde{\psi}(t) = 2\sinh(\log(\arcsinh(\E^t))),
\]
and derived a new approximation formula:
\begin{equation}
%I%=
\int_{-\infty}^{\infty} f(x) \D x
=\int_{-\infty}^{\infty} f(\tilde{\psi}(t))\tilde{\psi}'(t)\D t
\approx h\sum_{k=-M}^N f(\tilde{\psi}(kh))\tilde{\psi}'(kh).
\label{eq:Okayama-Hanada-formula}
\end{equation}
Furthermore, they theoretically showed that
the error of the modified formula, say $E_n$, is bounded by
\begin{equation}
 |E_n| \leq C \E^{-\sqrt{2\pi d \mu n}},
\label{eq:error-bound-form}
\end{equation}
where $\mu\geq \mu'$, and $C$ is explicitly given in a computable form.
This inequality not only shows that
the modified formula~\eqref{eq:Okayama-Hanada-formula}
can attain faster convergence than~\eqref{eq:Stenger-formula},
but it also indicates that the error can be rigorously estimated by the right-hand side.
This is useful for verified numerical integration.

The present work improves upon their results. Rather than
the conformal map $x=\psi(t)$ or $x=\tilde{\psi}(t)$,
we propose a new conformal map
\[
 x = \phi(t) = 2\sinh(\log(\log(1+\E^t))).
\]
The principle of this conformal map is derived from
the fact that the convergence rate is improved
by replacing $\arcsinh(\E^t)$ with $\log(1+\E^t)$
in some fields~\cite{Hara,OkaMachi,OkaShinKatsu}.
Consequently,
the following approximation formula is derived:
\begin{equation}
%I%=
\int_{-\infty}^{\infty} f(x) \D x
=\int_{-\infty}^{\infty} f(\phi(t))\phi'(t)\D t
\approx h\sum_{k=-M}^N f(\phi(kh))\phi'(kh).
\label{eq:Our-formula}
\end{equation}
Furthermore, as the main contribution of this work,
we provide two (general and special) theoretical error bounds
in the same form as~\eqref{eq:error-bound-form},
where $\mu$ does not change, but
a larger $d$ can be taken as compared to that in the previous studies.
This indicates that the improved formula~\eqref{eq:Our-formula}
can attain faster convergence than~\eqref{eq:Stenger-formula}
and~\eqref{eq:Okayama-Hanada-formula}.

The remainder of this paper is organized as follows.
First, existing and new theorems are summarized
in Section~\ref{sec:main_theorem}.
%The main idea of the modification is also explained in this section.
Then, numerical examples are provided in Section~\ref{sec:numer}.
Finally, proofs of the new theorems are given in Sections~\ref{sec:proofs}
and~\ref{sec:proofs-2}.

\section{Summary of existing and new results}
\label{sec:main_theorem}

Sections~\ref{subsec:Stenger-result}
and~\ref{subsec:Okayama-Hanada-result} describe the existing results,
and Sections~\ref{subsec:New-result-1}
and~\ref{subsec:New-result-2} describe the new results.
First, the relevant notations are introduced.
Let $\domD_d$ be a strip domain
defined by $\domD_d=\{\zeta\in\mathbb{C}:|\Im \zeta|< d\}$
for $d>0$.
% with $0<d<\pi/2$.
Furthermore, let
%$\domD_d^{-}=\domD_d \cap \{\zeta\in\mathbb{C}:\Re\zeta< 0\}$,
$\domD_d^{-}=\{\zeta\in\domD_d :\Re\zeta< 0\}$
and
%$\domD_d^{+}=\domD_d \cap \{\zeta\in\mathbb{C}:\Re\zeta\geq 0\}$.
$\domD_d^{+}=\{\zeta\in\domD_d :\Re\zeta\geq 0\}$.

\subsection{Error analysis of Stenger's formula}
\label{subsec:Stenger-result}

An error analysis for Stenger's formula~\eqref{eq:Stenger-formula}
can be expressed as the following theorem,
which is a restatement of an existing theorem~\cite[Theorem~1.5.16]{Stenger}.

\begin{theorem}[Okayama--Hanada~{\cite[Theorem~2.1]{Okayama-Hanada}}]
\label{thm:Stenger}
Assume that $f$ is analytic in $\psi(\domD_d)$ with $0<d<\pi/2$,
and that there exist positive constants $K$, $\alpha$, and $\beta$
such that
\begin{align}
 |f(z)|&\leq K |\E^{-z}|^{2\beta}\label{leq:f-Dd-plus-original}
\intertext{holds for all $z\in\psi(\domD_d^{+})$, and}
 |f(z)|&\leq K \frac{1}{|z|^{\alpha+1}}\label{leq:f-Dd-minus-original}
\end{align}
holds for all $z\in\psi(\domD_d^{-})$.
Let $\mu = \min\{\alpha,\beta\}$,
let $M$ and $N$ be defined as
\begin{equation}
\begin{cases}
M=n,\quad N=\lceil\alpha n/\beta\rceil
 & \,\,\,(\text{if}\,\,\,\mu = \alpha),\\
N=n,\quad M=\lceil\beta n/\alpha\rceil
 &  \,\,\,(\text{if}\,\,\,\mu = \beta),
\end{cases}
\label{eq:Def-MN}
\end{equation}
and let $h$ be defined as
\begin{equation}
h = \sqrt{\frac{2\pi d}{\mu n}}.
\label{eq:Def-h}
\end{equation}
Then, there exists a constant $C$ independent of $n$, such that
\[
\left|
\int_{-\infty}^{\infty}f(x)\D x
- h\sum_{k=-M}^N f(\psi(kh))\psi'(kh)
\right|
\leq C \E^{-\sqrt{2 \pi d \mu n}}.
\]
\end{theorem}

\subsection{Error bound for the formula by Okayama and Hanada}

\label{subsec:Okayama-Hanada-result}

Okayama and Hanada~\cite{Okayama-Hanada}
proposed the replacement of $\psi$ with $\tilde{\psi}$ in
Stenger's formula~\eqref{eq:Stenger-formula}. They also provided
the following theoretical error bound
for the modified formula~\eqref{eq:Okayama-Hanada-formula}.

\begin{theorem}[Okayama--Hanada~{\cite[Theorem~2.2]{Okayama-Hanada}}]
\label{thm:New}
Assume that $f$ is analytic in $\tilde{\psi}(\domD_d)$ with $0<d<\pi/2$,
and that there exist positive constants $K$, $\alpha$, and $\beta$
such that
\begin{align}
 |f(z)|&\leq K |\E^{-z}|^{\beta}
 \label{leq:f-Dd-plus}
\end{align}
holds for all $z\in\tilde{\psi}(\domD_d^{+})$, and
\begin{align}
 |f(z)|&\leq K\frac{1}{|4 + z^2|^{(\alpha+1)/2}}
 \label{leq:f-Dd-minus}
\end{align}
holds for all $z\in\tilde{\psi}(\domD_d^{-})$.
Let $\mu = \min\{\alpha,\beta\}$,
let $M$ and $N$ be defined as~\eqref{eq:Def-MN},
and let $h$ be defined as~\eqref{eq:Def-h}.
Then, it holds that
\[
\left|
\int_{-\infty}^{\infty}f(x)\D x
- h\sum_{k=-M}^N f(\tilde{\psi}(kh))\tilde{\psi}'(kh)
\right|
\leq K\left(\frac{2C_1}{1-\E^{-\sqrt{2\pi d \mu}}} + C_2\right)
 \E^{-\sqrt{2 \pi d \mu n}},
\]
where $C_1$ and $C_2$ are constants defined as
\begin{align*}
C_1&
=\frac{\gamma_d}{\alpha\arctan(\gamma_d)}
\left\{\frac{\gamma_d}{2}\left(1 + \frac{1}{\sin^2 1}\right)\right\}^{\alpha}
+\frac{(1 + \sigma^2)\sqrt{\gamma_d}}{\beta}
\left\{\frac{\sqrt{2}\E^{\sigma}}{\cos(d/2)}\right\}^{\beta},\\
%\label{def:C_1}\\
C_2&=\frac{1}{\alpha}
\left\{\frac{1}{2}\left(1+\frac{1}{\sin^2 1}\right)\right\}^{\alpha}
+\frac{1+\sigma^2}{\beta}\left(\frac{\E^{\sigma}}{2}\right)^{\beta},\\
%\label{def:C_2}
\end{align*}
where $\gamma_d=1/\cos(d)$ and $\sigma=1/\arcsinh(1)$.
\end{theorem}

In Theorem~\ref{thm:New},
the condition~\eqref{leq:f-Dd-plus-original} is
modified to~\eqref{leq:f-Dd-plus},
and the condition~\eqref{leq:f-Dd-minus-original} is
modified to~\eqref{leq:f-Dd-minus}.
The former constitutes the most significant difference,
because $\beta$ in Theorem~\ref{thm:New} can be
two times greater than that in Theorem~\ref{thm:Stenger},
while $\alpha$ remains unchanged.
Owing to the difference, $\mu$ in Theorem~\ref{thm:New}
may be greater than that in Theorem~\ref{thm:Stenger},
which affects the convergence rate
$\Order(\E^{-\sqrt{2\pi d \mu n}})$.

Another difference between
Theorems~\ref{thm:Stenger}
and~\ref{thm:New} lies in the constants
on the right-hand side of the inequalities.
% in front of $\E^{-\sqrt{2\pi d \mu n}}$.
All the constants in Theorem~\ref{thm:New} are explicitly revealed,
and the right-hand side can be computed to provide an error bound.
%In contrast, the constant $C$ in Theorem~\ref{thm:Stenger}
%is only claimed to exist.
This paper provides two error bounds
for the improved formula~\eqref{eq:Our-formula}
in the same manner as Theorem~\ref{thm:New}.

\subsection{General error bound for the proposed formula}
\label{subsec:New-result-1}

As a general case, we present the following error bound
for the improved formula~\eqref{eq:Our-formula}.
The proof is given in Section~\ref{sec:proofs}.

\begin{theorem}
\label{thm:New1}
Assume that $f$ is analytic in $\phi(\domD_d)$ with $0<d<\pi$,
and that there exist positive constants $K$, $\alpha$, and $\beta$
such that~\eqref{leq:f-Dd-plus}
holds for all $z\in\phi(\domD_d^{+})$, and~\eqref{leq:f-Dd-minus-original}
holds for all $z\in\phi(\domD_d^{-})$.
Let $\mu = \min\{\alpha,\beta\}$,
let $M$ and $N$ be defined as~\eqref{eq:Def-MN},
and let $h$ be defined as~\eqref{eq:Def-h}.
Then, it holds that
\[
\left|
\int_{-\infty}^{\infty}f(x)\D x
- h\sum_{k=-M}^N f(\phi(kh))\phi'(kh)
\right|
\leq K\left(\frac{2C_3}{1-\E^{-\sqrt{2\pi d \mu}}} + C_4\right)
 \E^{-\sqrt{2 \pi d \mu n}},
\]
where $C_3$ and $C_4$ are constants defined as
\begin{align}
C_3&
=\left(\frac{1}{\alpha+1}+\frac{1}{\alpha}\right)
\left\{\frac{\E c_d}{(1-\log 2)(\E - 1)}\right\}^{\alpha+1}
\frac{1+\{\log(2+c_d)\}^2}{\{\log(2+c_d)\}^2}(1+c_d)^2
+\frac{(1+\lambda^2)c_d}{\beta}
\left(\E^{\lambda} c_d\right)^{\beta},
%\frac{\E^{\beta/\log 2}c_d^{\beta+1}}{\beta}
%\left\{1+\frac{1}{(\log 2)^2}\right\},
\label{def:C_3}\\
C_4&=\frac{\E^{1/\pi^3}}{\alpha (1 - \log 2)^{\alpha+1}}
+\frac{1+\lambda^2}{\beta}\left(\E^{\lambda}\right)^{\beta},
%\frac{\E^{\beta/\log 2}}{\beta}
%\left\{1+\frac{1}{(\log 2)^2}\right\},
\label{def:C_4}
\end{align}
where $c_d = 1/\cos(d/2)$ and $\lambda=1/\log 2$.
\end{theorem}

The crucial difference between
Theorems~\ref{thm:New} and~\ref{thm:New1} is the upper bound of $d$;
$d<\pi/2$ in Theorem~\ref{thm:New}, whereas
$d<\pi$ in Theorem~\ref{thm:New1}.
This implies that in the new approximation~\eqref{eq:Our-formula}, $d$
may be greater than $d$ in the previous
approximation~\eqref{eq:Okayama-Hanada-formula}.
In this case, the convergence rate
$\Order(\E^{-\sqrt{2\pi d \mu n}})$ is improved
(note that $\mu$ is not changed between the two theorems).
%in the case of the new approximation~\eqref{eq:Our-formula}.

This difference in the range of $d$ originates from the conformal maps
$\tilde{\psi}$ and $\phi$.
By observing the derivatives of the functions
\[
 \tilde{\psi}'(\zeta)
= \frac{1 + \arcsinh^2(\E^\zeta)}{\sqrt{1+\E^{-2\zeta}}\arcsinh^2(\E^\zeta)},
\quad
\phi'(\zeta)
= \frac{1 + \{\log(1+\E^{\zeta})\}^2}{(1+\E^{-\zeta})\{\log(1+\E^{\zeta})\}^2},
\]
we see that $\tilde{\psi}'(\zeta)$
is not analytic at $\zeta=\pm\I(\pi/2)$,
and $\phi'(\zeta)$ is not analytic at $\zeta=\pm\I\pi$.
Accordingly, $f(\tilde{\psi}(\zeta))\tilde{\psi}'(\zeta)$ is
analytic at most $\domD_{\pi/2}$, and
$f(\phi(\zeta))\phi'(\zeta)$ is analytic at most $\domD_{\pi}$.
Therefore, the range of $d$ is $0<d<\pi/2$ in Theorem~\ref{thm:New}
and $0 < d < \pi$ in Theorem~\ref{thm:New1}.

\subsection{Special error bound for the proposed formula}
\label{subsec:New-result-2}

As a special case, restricting the range of $d$ to $0<d<(1+\pi)/2$,
we present the following error bound
for the improved formula~\eqref{eq:Our-formula}.
The proof is given in Section~\ref{sec:proofs-2}.

\begin{theorem}
\label{thm:New2}
Assume that $f$ is analytic in $\phi(\domD_d)$ with $0<d<(1+\pi)/2$,
and that there exist positive constants $K$, $\alpha$, and $\beta$
such that~\eqref{leq:f-Dd-plus}
holds for all $z\in\phi(\domD_d^{+})$, and
\begin{equation}
 |f(z)|\leq K \frac{1}{|4+z^2|^{1/2}|z|^{\alpha}}
\label{leq:f-Dd-minus-new}
\end{equation}
holds for all $z\in\phi(\domD_d^{-})$.
Let $\mu = \min\{\alpha,\beta\}$,
let $M$ and $N$ be defined as~\eqref{eq:Def-MN},
and let $h$ be defined as~\eqref{eq:Def-h}.
Then, it holds that
\[
\left|
\int_{-\infty}^{\infty}f(x)\D x
- h\sum_{k=-M}^N f(\phi(kh))\phi'(kh)
\right|
\leq K\left(\frac{2C_5}{1-\E^{-\sqrt{2\pi d \mu}}} + C_6\right)
 \E^{-\sqrt{2 \pi d \mu n}},
\]
where $C_5$ and $C_6$ are constants defined as
\begin{align}
C_5&
=\frac{1}{\alpha}\left\{\frac{\E c_d}{(1-\log 2)(\E - 1)}\right\}^{\alpha}
\frac{1 + c_d}{\log(2+c_d)}
+\frac{(1+\lambda^2)c_d}{\beta}
\left(\E^{\lambda} c_d\right)^{\beta},
%\frac{\E^{\beta/\log 2}c_d^{\beta+1}}{\beta}
%\left\{1+\frac{1}{(\log 2)^2}\right\},
\label{def:C_5}\\
C_6&=\frac{1}{\alpha (1 - \log 2)^{\alpha}}
+\frac{1+\lambda^2}{\beta}\left(\E^{\lambda}\right)^{\beta},
%\frac{\E^{\beta/\log 2}}{\beta}
%\left\{1+\frac{1}{(\log 2)^2}\right\},
\label{def:C_6}
\end{align}
where $c_d = 1/\cos(d/2)$ and $\lambda=1/\log 2$.
\end{theorem}

In this theorem, the upper bound of $d$ is $(1 + \pi)/2$,
which is smaller than that in Theorem~\ref{thm:New1} ($\pi$).
This is because the condition~\eqref{leq:f-Dd-minus-original}
is changed to~\eqref{leq:f-Dd-minus-new},
where $4+\{\phi(\zeta)\}^2$ (put $z=\phi(\zeta)$) has zero points at
$\zeta=\log(2\sin(1/2))\pm \I(1+\pi)/2$.
However,
%$(1 + \pi)/2$ is still greater than $\pi/2$
%in Theorems~\ref{thm:Stenger} and~\ref{thm:New}, and
the constants $C_5$ and $C_6$ are considerably smaller
than $C_3$ and $C_4$, respectively (comparing the first term).
Therefore, Theorem~\ref{thm:New2} is useful for attaining a sharp error bound
rather than a large upper bound of $d$.
It must be noted here that
$(1 + \pi)/2$ is still greater than $\pi/2$
in Theorems~\ref{thm:Stenger} and~\ref{thm:New}.

\section{Numerical examples}
\label{sec:numer}

This section presents the numerical results obtained in this study.
All the programs were written in C language with double-precision floating-point
arithmetic.
The following three integrals are considered:
\begin{align}
\int_{-\infty}^{\infty}\left\{\frac{1}{\sqrt{1 + (x/2)^2}+ 1-(x/2)}\right\}^2
\exp\left(-\frac{x}{2}-\sqrt{1+\left(\frac{x}{2}\right)^2}\right)\D x
&=3 - 4\E E_1(1),
\label{eq:example1}\\
%\int_{-\infty}^{\infty} \frac{1}{(4+x^2)^2 (1+\E^{\frac{\pi}{2}x})}\D x
%&=\frac{\pi}{32},
\int_{-\infty}^{\infty}\frac{1}{4+x^2}
\exp\left(-\frac{x}{2}-\sqrt{1+\left(\frac{x}{2}\right)^2}\right)\D x
&=\Ci(1)\sin 1 - \si(1)\cos 1,
\label{eq:example3}\\
\int_{-\infty}^{\infty} \frac{1}{2}\left(1+\frac{x}{\sqrt{4+x^2}}\right)\frac{1}{1+\E^{(\pi/2)x}}\D x
&=1.136877446810281077257\cdots,
\label{eq:example2}
\end{align}
where $E_1(x)$ is the exponential integral defined by
$E_1(x)=\int_1^{\infty}(\E^{-tx}/t)\D t$,
$\Ci(x)$ is the cosine integral defined by
$\Ci(x)=-\int_x^{\infty}(\cos t/t)\D t$,
and $\si(x)$ is the sine integral defined by
$\si(x)=-\int_x^{\infty}(\sin t/t)\D t$.
The third integral~\eqref{eq:example2} is taken from
the previous study~\cite{Okayama-Hanada}.

\begin{table}[htbp]
\caption{Parameters for the integral~\eqref{eq:example1}.}
\label{tbl:example1}
\centering
\begin{tabular}{lcccc}
\hline
                         & $\alpha$ & $\beta$ & $d$  & $K$ \\
\hline
Theorem~\ref{thm:Stenger}& $1$      & $1/2$   & $3/2$&   \\
Theorem~\ref{thm:New}    & $1$      &   $1$   & $3/2$& $1$ \\
Theorem~\ref{thm:New1}   & $1$      &   $1$   & $3$& $78$\\
Theorem~\ref{thm:New2}   & $1$      &   $1$   & $2$& $6/5$\\
\hline
\end{tabular}
\end{table}

\begin{table}[htbp]
\caption{Parameters for the integral~\eqref{eq:example3}.}
\label{tbl:example3}
\centering
\begin{tabular}{lcccc}
\hline
                         & $\alpha$ & $\beta$ & $d$  & $K$ \\
\hline
Theorem~\ref{thm:Stenger}& $1$      & $1/2$   & $3/2$&   \\
Theorem~\ref{thm:New}    & $1$      &   $1$   & $3/2$& $16/9$ \\
Theorem~\ref{thm:New1}   & $1$      &   $1$   & $2$& $215$\\
Theorem~\ref{thm:New2}   & $1$      &   $1$   & $2$& $39$\\
\hline
\end{tabular}
\end{table}

\begin{table}[htbp]
\caption{Parameters for the integral~\eqref{eq:example2}.}
\label{tbl:example2}
\centering
\begin{tabular}{lcccc}
\hline
                         & $\alpha$ & $\beta$ & $d$  & $K$ \\
\hline
Theorem~\ref{thm:Stenger}& $1$      & $\pi/4$   & $3/2$&    \\
Theorem~\ref{thm:New}    & $1$      & $\pi/2$   & $3/2$& $12$ \\
Theorem~\ref{thm:New1}   & $1$      & $\pi/2$   & $3/2$& $9$\\
Theorem~\ref{thm:New2}   & $1$      & $\pi/2$   & $3/2$& $9/2$\\
\hline
\end{tabular}
\end{table}

The integrand in~\eqref{eq:example1}
satisfies the assumptions of Theorems~\ref{thm:Stenger},
\ref{thm:New}, \ref{thm:New1}, and \ref{thm:New2}
with the parameters shown in Table~\ref{tbl:example1}.
In Theorem~\ref{thm:Stenger}, $K$ is not investigated
since $K$ is not used for computation.
In Theorems~\ref{thm:Stenger} and~\ref{thm:New},
$d$ is taken as $d=3/2$ since $d<\pi/2$.
In Theorem~\ref{thm:New1},
$d$ is taken as $d=3$ since $d<\pi$.
In Theorem~\ref{thm:New2},
$d$ is taken as $d=2$ since $d<(1+\pi)/2$.
The results are shown in Figure~\ref{fig:example1}.
As seen in the graph,
the proposed formula with $d=3$
shows the fastest convergence as compared to the others.
However, the corresponding error bound by Theorem~\ref{thm:New1}
is relatively large, because the constant $C_3$ in~\eqref{def:C_3} is large.
In contrast, Theorem~\ref{thm:New2} produces
a sharp error for the proposed formula with $d=2$,
although the convergence rate is slightly worse than that from Theorem~\ref{thm:New1}.

The integrand in~\eqref{eq:example3}
satisfies the assumptions of Theorems~\ref{thm:Stenger},
\ref{thm:New}, \ref{thm:New1}, and \ref{thm:New2}
with the parameters shown in Table~\ref{tbl:example3}.
In this case, $d$ must satisfy $d<(1+\pi)/2$
in Theorem~\ref{thm:New1},
due to the singular points of $1/(4+x^2)$.
The results are shown in Figure~\ref{fig:example3}.
As seen in the graph,
the proposed formula with $d=2$
shows the fastest convergence as compared to the others.
Note that in this example, Theorem~\ref{thm:New1} and Theorem~\ref{thm:New2}
have the same $d$ values,
and thus, their approximation formulas are exactly the same.
As for the error bound,
Theorem~\ref{thm:New2} produces a sharper error
than Theorem~\ref{thm:New1} in this case as well.

The integrand in~\eqref{eq:example2}
satisfies the assumptions of Theorems~\ref{thm:Stenger},
\ref{thm:New}, \ref{thm:New1}, and \ref{thm:New2}
with the parameters shown in Table~\ref{tbl:example2}.
In this case, $d$ must satisfy $d<\pi/2$
in both Theorems~\ref{thm:New1} and~\ref{thm:New2},
due to the singular points of $1/(1+\E^{(\pi/2)x})$.
The results are shown in Figure~\ref{fig:example2}.
As seen in the graph,
all formulas show a similar convergence rate,
mainly because all formulas use the same value of $d$.
%Note that in this example, Theorem~\ref{thm:New1} and Theorem~\ref{thm:New2}
%have the same $d$ values,
%and thus, their approximation formulas are exactly the same.
Approximation formulas of Theorem~\ref{thm:New1} and Theorem~\ref{thm:New2}
are exactly the same,
but Theorem~\ref{thm:New2} produces a sharper error
than Theorem~\ref{thm:New1} in this case as well.

\begin{figure}[htbp]
\centering
\input{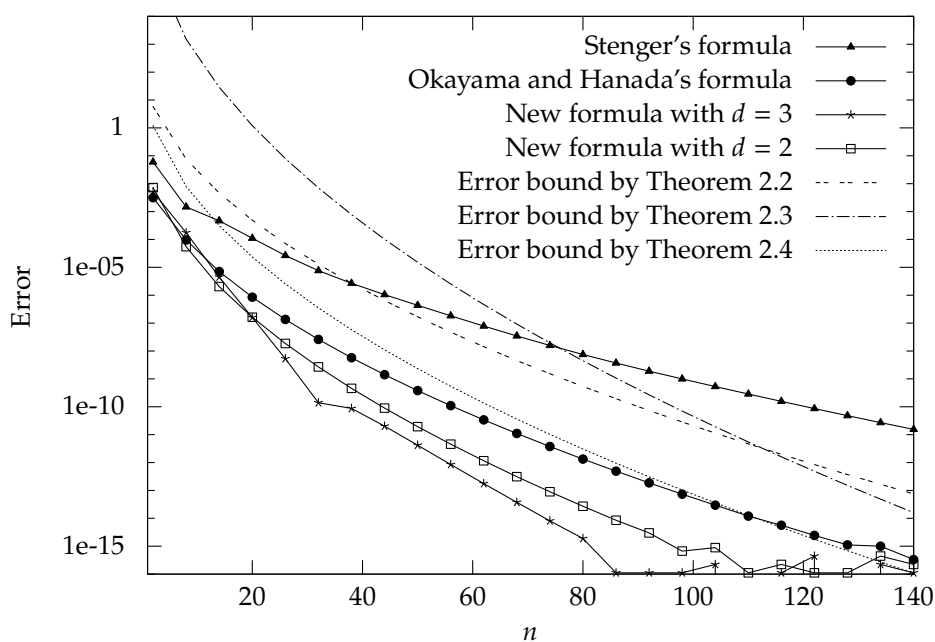}
\caption{Numerical results for~\eqref{eq:example1}.}
\label{fig:example1}
\end{figure}
\begin{figure}[htbp]
\centering
\input{example3.tex}
\caption{Numerical results for~\eqref{eq:example3}.}
\label{fig:example3}
\end{figure}
\begin{figure}[htbp]
\centering
\input{example2.tex}
\caption{Numerical results for~\eqref{eq:example2}.}
\label{fig:example2}
\end{figure}
%\begin{figure}[htpb]
%%{\centering
%\begin{minipage}{0.48\linewidth}
%\includegraphics[width=\linewidth]{example1.eps}
%\caption{Numerical results for~\eqref{eq:example1}.}\label{fig:example1}
%\end{minipage}
% \begin{minipage}{0.02\linewidth}
%  \mbox{ }
% \end{minipage}
%\begin{minipage}{0.48\linewidth}
%\includegraphics[width=\linewidth]{example2.eps}
%\caption{Numerical results for~\eqref{eq:example2}.}\label{fig:example2}
%\end{minipage}
%%}
%\end{figure}

\section{Proofs for Theorem~\ref{thm:New1}}
\label{sec:proofs}

This section presents the proof of Theorem~\ref{thm:New1}.
It is organized as follows.
In Section~\ref{subsec:sketch-proof},
the task is decomposed into two lemmas:
Lemmas~\ref{lem:bound-None} and~\ref{lem:bound-truncation-error}.
To prove these lemmas,
useful inequalities are presented in
Sections~\ref{subsec:real-ineq}, \ref{subsec:domDdplus}, \ref{subsec:domDdminus},
and~\ref{subsec:domDd}.
Following this,
Lemma~\ref{lem:bound-None} is proved in Section~\ref{subsec:discretization-error},
and Lemma~\ref{lem:bound-truncation-error} is proved in Section~\ref{subsec:truncation-error}.

\subsection{Sketch of the proof}
\label{subsec:sketch-proof}

Let $F(t)=f(\phi(t))\phi'(t)$.
The main strategy in the proof of Theorem~\ref{thm:New1} is
to split the error into two terms as follows:
\begin{align}
\left|
\int_{-\infty}^{\infty}f(x)\D x
- h\sum_{k=-M}^N f(\phi(kh))\phi'(kh)
\right|
&=\left|
\int_{-\infty}^{\infty} F(t)\D t
- h\sum_{k=-M}^N F(kh)
\right|\nonumber\\
&\leq \left|
\int_{-\infty}^{\infty}F(t)\D t
- h\sum_{k=-\infty}^{\infty} F(kh)
\right|
+
\left|
h\sum_{k=-\infty}^{-M-1} F(kh)
+h\sum_{k=N+1}^{\infty} F(kh)
\right|.
\label{eq:decompose-disc-trun}
\end{align}
The first and second terms are called the discretization error and
truncation error, respectively.
The following function space is important for bounding the discretization error.

\begin{definition}
\label{Def:Hone}
Let $\domD_d(\epsilon)$ be a rectangular domain defined
for $0<\epsilon<1$ by
\[
\domD_d(\epsilon)
= \{\zeta\in\mathbb{C}:|\Re\zeta|<1/\epsilon,\, |\Im\zeta|<d(1-\epsilon)\}.
\]
Then, $\mathbf{H}^1(\domD_d)$ denotes the family of all functions $F$
that are analytic in $\domD_d$ such that the norm $\mathcal{N}_1(F,d)$ is finite, where
\[
\mathcal{N}_1(F,d)
=\lim_{\epsilon\to 0}\oint_{\partial \domD_d(\epsilon)} |F(\zeta)||\D \zeta|.
\]
\end{definition}

For functions belonging to this function space,
the discretization error is estimated as follows.

\begin{theorem}[Stenger~{\cite[Theorem~3.2.1]{stenger93:_numer}}]
\label{thm:Stenger-disc}
Let $F\in\mathbf{H}^1(\domD_d)$. Then,
\[
\left|
\int_{-\infty}^{\infty}F(x)\D x - h\sum_{k=-\infty}^{\infty}F(kh)
\right|
\leq\frac{\mathcal{N}_1(F,d)}{1-\E^{-2\pi d/h}}\E^{-2\pi d/h}.
\]
\label{thm:bound-discretization-error}
\end{theorem}

In this paper, we show the following lemma,
which completes estimation of the discretization error.
The proof is given in Section~\ref{subsec:discretization-error}.

\begin{lemma}
%{Lemma 4.3}
\label{lem:bound-None}
Let the assumptions made in Theorem~\ref{thm:New1}
be fulfilled.
Then, the function $F(\zeta)=f(\phi(\zeta))\phi'(\zeta)$
belongs to $\mathbf{H}^1(\domD_d)$, and $\mathcal{N}_1(F,d)$ is bounded as
\[
 \mathcal{N}_1(F,d)
\leq
2 K C_3,
\]
where $C_3$ is a constant defined as~\eqref{def:C_3}.
\end{lemma}

In addition, we bound the truncation error as follows.
The proof is given in Section~\ref{subsec:truncation-error}.

\begin{lemma}
%{Lemma 4.4}
\label{lem:bound-truncation-error}
Let the assumptions made in Theorem~\ref{thm:New1}
be fulfilled.
Then, setting $F(\zeta)=f(\phi(\zeta))\phi'(\zeta)$, we have
\begin{align*}
\left|
h\sum_{k=-\infty}^{-M-1} F(kh)
+
h\sum_{k=N+1}^{\infty} F(kh)
\right|
\leq K C_4 \E^{-\mu n h},
\end{align*}
where $C_4$ is a constant defined as~\eqref{def:C_4}.
\end{lemma}

Setting $h$ as~\eqref{eq:Def-h}, the above estimates
(Theorem~\ref{thm:bound-discretization-error},
Lemmas~\ref{lem:bound-None}, and~\ref{lem:bound-truncation-error})
yield the desired result as
\begin{align*}
\left|
\int_{-\infty}^{\infty}f(x)\D x
- h\sum_{k=-M}^N f(\phi(kh))\phi'(kh)
\right|
&\leq \frac{2KC_3}{1-\E^{-2\pi d/h}}\E^{-2\pi d/h}
+ K C_4 \E^{-\mu n h}\\
&= K\left( \frac{2 C_3}{1-\E^{-\sqrt{2\pi d \mu n}}}
+ C_4\right)\E^{-\sqrt{2\pi d \mu n}}\\
&\leq K\left( \frac{2 C_3}{1-\E^{-\sqrt{2\pi d \mu}}}
+ C_4\right)\E^{-\sqrt{2\pi d \mu n}}.
\end{align*}
This completes the proof of Theorem~\ref{thm:New1}.

\subsection{Useful inequalities on $\mathbb{R}$}
\label{subsec:real-ineq}

We prepare two lemmas here.
%which is used in Section~\ref{subsec:domDdplus}.

\begin{lemma}[Okayama et al.~{\cite[Lemma 4.7]{OkaShinKatsu}}]
%{Lemma 4.8}
%\label{lem:exp-asinh-SE}
We have
%and $x\in\mathbb{R}$, we have
\begin{align}
\left|
\frac{\log(1+\E^{x})}{1+\log(1+\E^{x})}
\cdot\frac{1+\E^{x}}{\E^{x}}
\right|
&\leq 1 \quad(x\in\mathbb{R}).
\label{ineq:exp-log-real}
%\frac{1}{\E^{x}+\sqrt{1+\E^{2x}}}
%&\leq\frac{1}{1+\E^{x}}.
%\label{ineq:exp-asinh-real}
\end{align}
\end{lemma}

\begin{lemma}
We have
\begin{equation}
\arccos\left(\frac{t}{2}\right)\geq \sqrt{2 - t}
\quad (0\leq t\leq 2).
\label{leq:arccos-bound}
\end{equation}
\end{lemma}
\begin{proof}
Integrating both sides of the obvious inequality
\[
 2 - 2\cos w\geq 0\quad (w\geq 0),
\]
we have
\[
 \int_0^{v}2(1 - \cos w)\D w = 2 v - 2\sin v \geq 0
\quad (v\geq 0).
\]
In the same manner, integrating both sides of the above inequality,
we have
\[
 \int_0^{u}2(v - \sin v)\D v = u^2 + 2\cos u - 2 \geq 0
\quad (u\geq 0).
\]
Here, putting $u=\arccos(t/2)$, we rewrite the inequality as
\[
 \arccos^2 \left(\frac{t}{2}\right) \geq  2 - t\quad (0\leq t\leq 2),
\]
which is equivalent to the desired inequality~\eqref{leq:arccos-bound}.
\end{proof}

\subsection{Useful inequalities on $\domD_d^{+}$}
\label{subsec:domDdplus}

We prepare three lemmas here. Note that $\overline{\domD}$
denotes the closure of $\domD$.

\begin{lemma}
%{Lemma base}
It holds for all $\zeta\in \overline{\domD_{\pi}^{+}}$ that
\begin{equation}
\left|
\frac{1}{\log(1+\E^{\zeta})}
\right|\leq \frac{1}{\log 2}.
\label{leq:Lemma-base}
\end{equation}
\end{lemma}
\begin{proof}
Let $\zeta=x + \I y$ where $x$ and $y$ are real numbers with
$x\geq 0$ and $|y|\leq\pi$.
By the definition of $\log z$, it holds that
\[
\left|
\frac{1}{\log(1+\E^{\zeta})}
\right|^2
=\left|
\frac{1}{\log|1+\E^{\zeta}| + \I \arg(1+\E^{\zeta})}
\right|^2
=\frac{1}{\left\{\log|1+\E^{x+\I y}|\right\}^2 + \left\{\arg(1+\E^{x+\I y})\right\}^2}.
\]
Since $|1+\E^{x+\I y}|$ and $|\arg(1+\E^{x+\I y})|$
monotonically increase with respect to $x$,
% d/dx {(1 + e^x cos y)^2 + (e^x sin y)^2) } = 2 e^x (e^x cos y)
% >= 2 (1 + cos y) >= 0
% d/dx {arctan(e^x sin y / (1 + e^x cos y)} = sin y/(2(cosh x + cos y))
% >= 0 (y >= 0), <= 0 (y <= 0)
we have
\[
 \frac{1}{\left\{\log|1+\E^{x+\I y}|\right\}^2 + \left\{\arg(1+\E^{x+\I y})\right\}^2}
\leq \frac{1}{\left\{\log|1+\E^{0+\I y}|\right\}^2 + \left\{\arg(1+\E^{0+\I y})\right\}^2}.
\]
Furthermore, using
\begin{align*}
\log|1 + \E^{\I y}|
&=\log\sqrt{(1+\cos y)^2 + \sin^2 y}
=\log\sqrt{4\cos^2\left(\frac{y}{2}\right)}
=\log\left(2\cos\frac{y}{2}\right),\\
\arg(1+\E^{\I y})
&=\arctan\left(\frac{\sin y}{1+\cos y}\right)
=\arctan\left(\tan\frac{y}{2}\right)
=\frac{y}{2},
%\label{log-1-cos-sin-Ddplus}
\end{align*}
and putting $t=2\cos(y/2)$, we have
\[
 \frac{1}{\left\{\log|1+\E^{\I y}|\right\}^2 + \left\{\arg(1+\E^{\I y})\right\}^2}
=\frac{1}{\left\{\log\left(2\cos(y/2)\right)\right\}^2 + (y/2)^2}
=\frac{1}{\left(\log t\right)^2 + \arccos^2(t/2)}.
\]
From $0\leq t\leq 2$ and~\eqref{leq:arccos-bound}, we have
\[
\frac{1}{\left(\log t\right)^2 + \arccos^2(t/2)}
\leq \frac{1}{(\log t)^2 + \{\sqrt{2 - t}\}^2}
= q(t),
\]
where
\[
 q(t)=\frac{1}{(\log t)^2 + 2 - t}.
\]
Since $\log t\leq t - 1$, it holds that
\[
 q'(t)= \frac{t - 2\log t}{t\{(\log t)^2 + 2 - t\}^2}
\geq  \frac{t - 2(t-1)}{t\{(\log t)^2 + 2 - t\}^2}
=\frac{2 - t}{t\{(\log t)^2 + 2 - t\}^2}\geq 0.
\]
Therefore, $q(t)$ monotonically increases, from which we have
$q(t)\leq q(2)=1/(\log 2)^2$.
This completes the proof.
\end{proof}

\begin{lemma}
%{Lemma 4.5}
It holds for all $\zeta\in \overline{\domD_{\pi}^{+}}$ that
\begin{equation}
\left|
\E^{1/\log(1+\E^{\zeta})}
\right|\leq \E^{1/\log 2}.
\label{leq:Lemma45}
\end{equation}
\end{lemma}
\begin{proof}
Using~\eqref{leq:Lemma-base}, we have
\[
 \left|\E^{1/\log(1+\E^{\zeta})}\right|
\leq \E^{\left|1/\log(1+\E^{\zeta})\right|}
\leq \E^{1/\log 2}.
\]
This completes the proof.
\end{proof}

\begin{lemma}
%{Lemma 4.6}
It holds for all $\zeta\in \overline{\domD_{\pi}^{+}}$ that
\begin{equation}
\left|
\frac{1+\{\log(1+\E^{\zeta})\}^2}{\{\log(1+\E^{\zeta})\}^2}
\right|\leq 1 + \frac{1}{(\log 2)^2}.
\label{leq:Lemma46}
\end{equation}
\end{lemma}
\begin{proof}
Using~\eqref{leq:Lemma-base}, we have
\[
\left|
\frac{1+\{\log(1+\E^{\zeta})\}^2}{\{\log(1+\E^{\zeta})\}^2}
\right|
=\left|
1 + \frac{1}{\{\log(1+\E^{\zeta})\}^2}
\right|
\leq
1 + \left|\frac{1}{\{\log(1+\E^{\zeta})\}^2}\right|
\leq 1 + \frac{1}{(\log 2)^2}.
\]
This completes the proof.
\end{proof}

\subsection{Useful inequality on $\domD_d^{-}$}
\label{subsec:domDdminus}

We prepare the following lemma here.

\begin{lemma}
%{Lemma 4.7}
It holds for all $\zeta\in \overline{\domD_{\pi}^{-}}$ that
\begin{equation}
\frac{1}{|-1 +\log(1+\E^{\zeta})|}
\leq \frac{1}{1 - \log 2}.
\label{leq:func-bound-Ddminus}
\end{equation}
\end{lemma}
\begin{proof}
%Let $g(\zeta)=1/(-1 + \log(\E^{\zeta}))$.
%Since $g$ is analytic in $\domD_{\pi}^{-}$ and continuous
%on $\overline{\domD_{\pi}^{-}}$,
%by the maximum modulus principle,
%$|g(\zeta)|$ has its maximum
%on the boundary of $\domD_{\pi}^{-}$.
%We only consider the following three cases:
%$\zeta\to \I \pi$,
%$\zeta=x+\I \pi$ $(x < 0)$, and $\zeta=\I y$ $(0\leq y < \pi)$.
%A similar proof holds for other cases.
%
%First, consider the case $\zeta\to \I \pi$.
By the definition of $\log z$, it holds that
\[
\frac{1}{|-1 +\log(1+\E^{\zeta})|}
=\frac{1}{|-1+\log|1+\E^{\zeta}| + \I\arg(1+\E^{\zeta})|}
\leq \frac{1}{|-1+\log|1+\E^{\zeta}| + 0|}.
\]
Let $\zeta=x + \I y$ where $x$ and $y$ are real numbers with
$x < 0$ and $|y|\leq\pi$.
Then, we have
%Using~\eqref{ineq:exp-1-plus}, we have
\[
 \left|1+\E^{\zeta}\right|\leq
1 + \left|\E^{\zeta}\right|
=1 + \left|\E^{x+\I y}\right|
=1 + \E^x
< 1 + \E^0
< \E,
%\left|1+\E^{x}\cos y + \I \E^{x}\sin y\right|\\
%&=\sqrt{(1+\E^x\cos y)^2+(\E^x\sin y)^2}\\
%&=\sqrt{(1+\E^x)^2-4\E^x\sin^2(y/2)}\\
%&\leq 1 + \E^x \\
%& < 1 + \E^0\\
%& < \E,
\]
from which we have $\log|1+\E^{\zeta}| < 1$.
Therefore, it holds that
\[
%\frac{1}{|-1 +\log(1+\E^{\zeta})|}
%\leq
 \frac{1}{|-1+\log|1+\E^{\zeta}||}
=\frac{1}{1 - \log|1+\E^{\zeta}|},
\]
which is further bounded as
\begin{align*}
\frac{1}{1 - \log|1+\E^{\zeta}|}
%=\frac{1}{1 - \log\sqrt{(1+\E^x)^2-4\E^x\sin^2(y/2)}}
\leq \frac{1}{1 - \log(1 + |\E^{\zeta}|)}
= \frac{1}{1 - \log(1+\E^x)}
\leq \frac{1}{1-\log(1+\E^0)}
=\frac{1}{1 - \log 2}.
\end{align*}
This completes the proof.
\end{proof}

\subsection{Useful inequalities on $\domD_d$}
\label{subsec:domDd}

We prepare four lemmas here.

\begin{lemma}[Okayama et al.~{\cite[Lemma 4.6]{OkaShinKatsu}}]
%{Lemma 4.8}
%\label{lem:exp-asinh-SE}
It holds for all $\zeta\in\overline{\domD_{\pi}}$ that
%and $x\in\mathbb{R}$, we have
\begin{align}
\left|
\frac{\log(1+\E^{\zeta})}{1+\log(1+\E^{\zeta})}
\cdot\frac{\E^{-l}+\E^{\zeta}}{\E^{\zeta}}
\right|
&\leq 1,
\label{ineq:exp-log-complex}
%\frac{1}{\E^{x}+\sqrt{1+\E^{2x}}}
%&\leq\frac{1}{1+\E^{x}}.
%\label{ineq:exp-asinh-real}
\end{align}
where $l=\log(\E/(\E - 1))$.
\end{lemma}

\begin{lemma}[Okayama et al.~{\cite[Lemma 4.21]{Okayama-et-al}}]
%{Lemma 4.9}~
%\label{lem:1-exp-bound}
For all $x\in\mathbb{R}$ and $y\in(-\pi,\pi)$,
putting $\zeta=x+\I y$, we have
\begin{align}
\frac{1}{|1+\E^{\zeta}|}
&\leq \frac{1}{(1+\E^x)\cos(y/2)},
\label{ineq:exp-1-plus}\\
\frac{1}{|1+\E^{-\zeta}|}
&\leq\frac{1}{(1+\E^{-x})\cos(y/2)}.
\label{ineq:exp-1-minus}
\end{align}
\end{lemma}

\begin{lemma}[Three lines lemma, cf.~{\cite[p.\ 133]{Stein-Shakarchi}}]
\label{lem:three-lines}
Let $g$ be analytic and bounded in $\domD_d$
and continuous on $\overline{\domD_d}$.
Let $M_g(y) = \sup_{x\in\mathbb{R}}|g(x+\I y)|$.
Then, we have
\[
 \{M_g(y)\}^{2d} \leq \{M_g(-d)\}^{d-y} \{M_g(d)\}^{y+d}
\quad (-d\leq y\leq d).
\]
\end{lemma}

\begin{lemma}
%{Lemma 4.10}
Let $d$ be a constant satisfying $0<d<\pi$.
For all $\zeta\in \domD_{d}$ and $x\in\mathbb{R}$, we have
\begin{align}
\left|
\frac{1+\{\log(1+\E^{\zeta})\}^2}{(1+\E^{-\zeta})^2\{\log(1+\E^{\zeta})\}^2}
\right|
&\leq \frac{1+\{\log(2+c_d)\}^2}{\{\log(2+c_d)\}^2}(1+c_d)^2,
 \label{leq:Dd-func-bound} \\
\frac{1+\{\log(1+\E^{x})\}^2}{(1+\E^{-x})^2\{\log(1+\E^{x})\}^2}
&\leq \E^{1/\pi^3},
\label{leq:real-func-bound}
\end{align}
where $c_d=1/\cos(d/2)$.
\end{lemma}
\begin{proof}
First, consider~\eqref{leq:real-func-bound}, which is proved by showing
\[
 p(t) = \frac{1+t^2}{t^2}(1-\E^{-t})^2 \leq \E^{1/\pi^3}
%\label{leq:real-func-bound-shown}
\]
for all $t> 0$ (put $t=\log(1+\E^x)$). The derivative of $p(t)$ is
expressed as
\[
 p'(t) = -\frac{2(\E^t - 1)(\E^t - t^3 - t - 1)}{t^3\E^{2t}}.
\]
Let $\kappa$ be a value that satisfies
$p'(\kappa)=0$ and $\log(108)<\kappa<\log(109)$,
i.e., $p(t)$ has its maximum at $t=\kappa$.
Using $\E^{\kappa}=\kappa^3+\kappa+1$, we have
\[
 p(\kappa)=\frac{1+\kappa^2}{\kappa^2}
\left(\frac{(\kappa^3+\kappa+1) - 1}{\E^{\kappa}}\right)^2
=\frac{(1+\kappa^2)^3}{\E^{2\kappa}}.
\]
Since the function $q(x)=(1+x^2)^3/\E^{2x}$
monotonically decreases for $x\geq (3+\sqrt{5})/2$,
$q(\log(109))<q(\kappa)<q(\log(108))$ holds
(note that $\log(108)>(3+\sqrt{5})/2$).
Thus, it holds that
\[
 1 < q(\log(109)) < p(\kappa) = q(\kappa) < q(\log(108)) < \E^{1/\pi^3}.
\]

Next, we show~\eqref{leq:Dd-func-bound}. Let
\[
 g(\zeta)=
\frac{1+\{\log(1+\E^{\zeta})\}^2}{(1+\E^{-\zeta})^2\{\log(1+\E^{\zeta})\}^2}.
\]
Since the function $g(\zeta)$
is analytic and bounded in $\domD_d$ and continuous
on $\overline{\domD_d}$,
by Lemma~\ref{lem:three-lines},
we obtain~\eqref{leq:Dd-func-bound} if we show the following two inequalities:
\begin{align*}
M_g(d) \leq \frac{1+\{\log(2+c_d)\}^2}{\{\log(2+c_d)\}^2}(1+c_d)^2,
\quad
M_g(-d) \leq \frac{1+\{\log(2+c_d)\}^2}{\{\log(2+c_d)\}^2}(1+c_d)^2,
\end{align*}
where $M_g(y) = \sup_{x\in\mathbb{R}}|g(x+\I y)|$.
We show only the first one, because the second one is also shown
in the same way.
Putting $\xi=\log(1+\E^{x + \I d})$, $g(x + \I d)=p(\xi)$ holds,
%where the function $p$ is introduced in~\eqref{leq:real-func-bound-shown}.
and thus, in what follows we prove
\begin{equation}
|p(\xi)|\leq \frac{1+\{\log(2+c_d)\}^2}{\{\log(2+c_d)\}^2}(1+c_d)^2.
\label{leq:Dd-func-bound-target}
\end{equation}
We consider the following two cases: (a) $|\xi|\leq \log(2+c_d)$
and (b) $|\xi|>\log(2+c_d)$.
In case (a), we have
\begin{align*}
|p(\xi)|
=\left|\frac{1+\xi^2}{\xi^2}
\left(
-\sum_{k=1}^{\infty}\frac{(-\xi)^k}{k!}
\right)^2
\right|
\leq (1+|\xi|^2)
\left(
 \sum_{k=1}^{\infty}\frac{|\xi|^{k-1}}{k!}
\right)^2
=\frac{1+|\xi|^2}{|\xi|^2}
\left(
\sum_{n=1}^{\infty}\frac{|\xi|^k}{k!}
\right)^2
=\frac{1+|\xi|^2}{|\xi|^2}\left(\E^{|\xi|}-1\right)^2.
%=p(-|\xi|).
\end{align*}
Here, if we put $q(x)=(1+x^2)(\E^x - 1)^2/x^2$,
then we have $q'(x)=2(\E^x - 1)r(x)/x^3$, where $r(x)=1+\E^x(x^3+x-1)$.
Since $r'(x)=x\E^x \{(x+1)^2+x\}\geq 0$ for $x\geq 0$,
$r(x)$ monotonically increases for $x\geq 0$.
Therefore, $r(x)\geq r(0)=0$ holds, from which we have
$q'(x)\geq 0$ for $x\geq 0$, i.e., $q(x)$ monotonically increases
for $x\geq 0$.
Thus, from $|\xi|\leq \log(2+c_d)$, we have~\eqref{leq:Dd-func-bound-target}
as
\[
 |p(\xi)|\leq q(|\xi|)
 \leq
 \frac{1+\{\log(2+c_d)\}^2}{\{\log(2+c_d)\}^2}(\E^{\log(2+c_d)}-1)^2
=\frac{1+\{\log(2+c_d)\}^2}{\{\log(2+c_d)\}^2}(1+c_d)^2.
\]
In case (b), from~\eqref{ineq:exp-1-plus}, it holds that
\begin{equation}
 \Re(\xi)=\Re(\log(1+\E^{x+\I d}))
=\log|1+\E^{x+\I d}|\geq \log[(1 + \E^x)\cos(d/2)]
\geq \log(\cos(d/2)).
\label{eq:bound-re-log}
\end{equation}
Using this, we have
\[
 |p(\xi)|\leq \frac{1+|\xi|^2}{|\xi|^2}(1 + |\E^{-\xi}|)^2
=\frac{1+|\xi|^2}{|\xi|^2}(1 + \E^{-\Re(\xi)})^2
\leq \frac{1+|\xi|^2}{|\xi|^2}(1 + \E^{-\log(\cos(d/2))})^2
= \frac{1+|\xi|^2}{|\xi|^2}(1 + c_d)^2.
\]
Furthermore, since $(1 + x^2)/x^2$ decreases monotonically for $x>0$,
we have~\eqref{leq:Dd-func-bound-target}.
This completes the proof.
\end{proof}

\subsection{Estimation of the discretization error (proof of Lemma~\ref{lem:bound-None})}
\label{subsec:discretization-error}

Lemma~\ref{lem:bound-None} is shown as follows.

\begin{proof}
Let $F(\zeta)=f(\phi(\zeta))\phi'(\zeta)$.
Since $f$ is analytic in $\phi(\domD_d)$,
$f(\phi(\cdot))$ is analytic in $\domD_d$.
In addition, since $\phi'$ is analytic in $\domD_{\pi}$,
$F$ is analytic in $\domD_d$ (note that $d<\pi$).
Therefore, the remaining task is to show
$\mathcal{N}_1(F,d)\leq 2KC_3$.
From~\eqref{leq:f-Dd-plus},
by using~\eqref{leq:Lemma45} and~\eqref{leq:Lemma46},
it holds for all $\zeta\in\domD_d^{+}$ that
\begin{align}
|F(\zeta)|&\leq K |\E^{-\phi(\zeta)}|^{\beta} |\phi'(\zeta)|\nonumber\\
&= K
\left|\E^{1/\log(1+\E^{\zeta})}\right|^{\beta}
\left|\frac{1}{1+\E^{\zeta}}\right|^{\beta}
\frac{|1+\{\log(1+\E^{\zeta})\}^2|}{|1+\E^{-\zeta}||\log(1+\E^{\zeta})|^2}
\nonumber\\
%&\leq K
%\left|\E^{1/\arcsinh(1)}\right|^{\beta}
%\left|\frac{\sqrt{2}}{1+\E^{\zeta}}\right|^{\beta}
%\frac{1}{|1+\E^{-2\zeta}|^{1/2}}
%\left(1+\frac{1}{\arcsinh^2(1)}\right)\\
&\leq K \left(\E^{1/\log 2}\right)^{\beta}
\frac{1}{|1+\E^{\zeta}|^{\beta}}
\frac{1}{|1+\E^{-\zeta}|}
\left\{1+\frac{1}{(\log 2)^2}\right\}.
\label{leq:bound-F-Dd-plus}
\end{align}
Furthermore, from~\eqref{leq:f-Dd-minus},
by using~\eqref{leq:func-bound-Ddminus}, \eqref{ineq:exp-log-complex},
and~\eqref{leq:Dd-func-bound},
it holds for all $\zeta\in\domD_d^{-}$ that
\begin{align}
|F(\zeta)|
&\leq K \frac{1}{|\phi(\zeta)|^{\alpha+1}}
\left|\phi'(\zeta)\right|\nonumber \\
&= K \left|\frac{\log(1+\E^{\zeta})}{1+\log(1+\E^{\zeta})}\right|^{\alpha+1}
\frac{|1+\E^{-\zeta}|}{|-1+\log(1+\E^{\zeta})|^{\alpha+1}}
\frac{|1+\{\log(1+\E^{\zeta})\}^2|}{|1+\E^{-\zeta}|^2|\log(1+\E^{\zeta})|^2}
\nonumber \\
&\leq K \frac{1}{|1+\E^{-\zeta-l}|^{\alpha+1}}
\frac{1 + |\E^{-\zeta}|}{(1-\log 2)^{\alpha+1}}
\frac{1 + \{\log(2+c_d)\}^2}{\{\log(2+c_d)\}^2}(1+c_d)^2,
\label{leq:bound-F-Dd-minus}
\end{align}
where $c_d = 1/\cos(d/2)$ and $l=\log(\E/(\E - 1))$.
By definition, $\mathcal{N}_1(F,d)$ is expressed as
\begin{align}
&\mathcal{N}_1(F,d)\nonumber\\
&=\lim_{\epsilon\to 0}
\left\{
\int_{-1/\epsilon}^{1/\epsilon}|F(x-\I d)|\D x
+\int_{-d(1-\epsilon)}^{d(1-\epsilon)} |F(1/\epsilon + \I y)|\D y
+\int_{-1/\epsilon}^{1/\epsilon}|F(x+\I d)|\D x
+\int_{-d(1-\epsilon)}^{d(1-\epsilon)} |F(-1/\epsilon + \I y)|\D y
\right\}.
\label{eq:def-None}
\end{align}
Using~\eqref{ineq:exp-1-plus}, \eqref{ineq:exp-1-minus},
and~\eqref{leq:bound-F-Dd-plus},
we can bound the second term as
\begin{align*}
\int_{-d(1-\epsilon)}^{d(1-\epsilon)}|F(1/\epsilon + \I y)|\D y
&\leq K \left(\E^{1/\log 2}\right)^{\beta}\left(1 + \frac{1}{(\log 2)^2}\right)
\int_{-d(1-\epsilon)}^{d(1-\epsilon)}
\frac{1}{|1+\E^{1/\epsilon+\I y}|^{\beta}|1+\E^{-(1/\epsilon+\I y)}|}\D y\\
&\leq \frac{K \left(\E^{1/\log 2}\right)^{\beta}\left(1 + \frac{1}{(\log 2)^2}\right)}
           {(1+\E^{1/\epsilon})^{\beta}(1+\E^{-1/\epsilon})}
\int_{-d(1-\epsilon)}^{d(1-\epsilon)}
\frac{1}{\cos^{\beta}(y/2)\cos(y/2)}\D y,
\end{align*}
from which we have
\[
 \lim_{\epsilon\to 0}\int_{-d(1-\epsilon)}^{d(1-\epsilon)}|F(1/\epsilon + \I y)|\D y
= 0.
\]
In the same manner,
with regard to the fourth term of~\eqref{eq:def-None},
using~\eqref{ineq:exp-1-plus}, \eqref{ineq:exp-1-minus},
and~\eqref{leq:bound-F-Dd-minus},
we have
\[
 \lim_{\epsilon\to 0}\int_{-d(1-\epsilon)}^{d(1-\epsilon)}|F(-1/\epsilon + \I y)|\D y
= 0.
\]
Therefore, $\mathcal{N}_1(F,d)$ is expressed as
\begin{align*}
\mathcal{N}_1(F,d)
&=\int_{-\infty}^{\infty}|F(x-\I d)|\D x
 +\int_{-\infty}^{\infty}|F(x+\I d)|\D x\\
&=\int_{-\infty}^{0}|F(x-\I d)| \D x
 +\int_{0}^{\infty}|F(x-\I d)|\D x
 +\int_{-\infty}^{0}|F(x+\I d)|\D x
 +\int_{0}^{\infty}|F(x+\I d)|\D x.
%\label{eq:None-rewrite}
\end{align*}
With regard to the first term,
using~\eqref{ineq:exp-1-plus}, \eqref{ineq:exp-1-minus},
and~\eqref{leq:bound-F-Dd-minus},
we have
\begin{align*}
\int_{-\infty}^{0}|F(x-\I d)|\D x
&\leq \frac{K}{(1- \log 2)^{\alpha+1}}
\frac{1+\{\log(2+c_d)\}^2}{\{\log(2+c_d)\}^2}(1+c_d)^2
\int_{-\infty}^{0}
\frac{1 + |\E^{-x+\I d}|}{|1+\E^{-x-l+\I d}|^{\alpha+1}}
\D x\\
&\leq \frac{K}{(1- \log 2)^{\alpha+1}}
\frac{1+\{\log(2+c_d)\}^2}{\{\log(2+c_d)\}^2}(1+c_d)^2
\int_{-\infty}^{0}
\frac{1 + |\E^{-x+\I d}|}{(1+\E^{-x-l})^{\alpha+1}\cos^{\alpha+1}(d/2)}
\D x\\
&=\frac{K c_d^{\alpha+1}}{(1- \log 2)^{\alpha+1}}
\frac{1+\{\log(2+c_d)\}^2}{\{\log(2+c_d)\}^2}(1+c_d)^2
\int_{-\infty}^{0}
\frac{1 + \E^{-x}}{(1+\E^{-x-l})^{\alpha+1}}
\D x.
\end{align*}
The integral is further bounded as
\begin{align*}
\int_{-\infty}^{0}
\frac{1 + \E^{-x}}{(1+\E^{-x-l})^{\alpha+1}}
\D x
&=\int_{-\infty}^{0}
\left(
\frac{\E^{(\alpha+1)x}}{(\E^x+\E^{-l})^{\alpha+1}}
+\frac{\E^{\alpha x}}{(\E^x+\E^{-l})^{\alpha+1}}
\right)
\D x\\
&\leq
\int_{-\infty}^{0}
\left(
\frac{\E^{(\alpha+1)x}}{(0+\E^{-l})^{\alpha+1}}
+\frac{\E^{\alpha x}}{(0+\E^{-l})^{\alpha+1}}
\right)
\D x\\
&=\left(\frac{\E}{\E - 1}\right)^{\alpha+1}
\left(\frac{1}{\alpha+1}+\frac{1}{\alpha}\right).
\end{align*}
In the same manner, the third term is bounded as
\[
\int_{-\infty}^{0}|F(x+\I d)|\D x
\leq
\frac{K c_d^{\alpha+1}}{(1- \log 2)^{\alpha+1}}
\frac{1+\{\log(2+c_d)\}^2}{\{\log(2+c_d)\}^2}(1+c_d)^2
\left(\frac{\E}{\E - 1}\right)^{\alpha+1}
\left(\frac{1}{\alpha+1}+\frac{1}{\alpha}\right).
\]
With regard to the second term,
using~\eqref{ineq:exp-1-plus}, \eqref{ineq:exp-1-minus},
and~\eqref{leq:bound-F-Dd-plus},
we have
\begin{align*}
\int_{0}^{\infty}|F(x-\I d)|\D x
&\leq K \E^{\beta/\log 2}\left\{1 + \frac{1}{(\log 2)^2}\right\}
\int_{0}^{\infty}
\frac{1}{|1+\E^{x-\I d}|^{\beta}|1+\E^{-x+\I d}|}
\D x\\
&\leq  K \E^{\beta/\log 2}\left\{1 + \frac{1}{(\log 2)^2}\right\}
\int_{0}^{\infty}
\frac{1}{(1+\E^x)^{\beta}(1+\E^{-x})\cos^{\beta+1}(d/2)}
\D x\\
&=K \E^{\beta/\log 2}\left\{1 + \frac{1}{(\log 2)^2}\right\} c_d^{\beta+1}
\int_{0}^{\infty}
\frac{\E^{-\beta x}}{(1+\E^{-x})^{\beta+1}}
\D x\\
&\leq K \E^{\beta/\log 2}\left\{1 + \frac{1}{(\log 2)^2}\right\} c_d^{\beta+1}
\int_{0}^{\infty}
\frac{\E^{-\beta x}}{(1+0)^{\beta+1}}
\D x\\
&=K \E^{\beta/\log 2}\left\{1 + \frac{1}{(\log 2)^2}\right\}
\frac{c_d^{\beta+1}}{\beta}.
\end{align*}
In the same manner, the fourth term is bounded as
\[
\int_{0}^{\infty}|F(x+\I d)|\D x
\leq K \E^{\beta/\log 2}\left\{1 + \frac{1}{(\log 2)^2}\right\}
\frac{c_d^{\beta+1}}{\beta}.
\]
Thus, we have $\mathcal{N}_1(F,d)\leq 2KC_3$.
\end{proof}

\subsection{Estimation of the truncation error (proof of Lemma~\ref{lem:bound-truncation-error})}
\label{subsec:truncation-error}

Lemma~\ref{lem:bound-truncation-error} is shown as follows.

\begin{proof}
Let $F(t)=f(\phi(t))\phi'(t)$.
From~\eqref{leq:f-Dd-plus},
by using~\eqref{leq:Lemma45} and~\eqref{leq:Lemma46},
%and~\eqref{ineq:exp-asinh-real},
it holds for all $t\geq 0$ that
\begin{align*}
|F(t)|&\leq K \left(\E^{-\phi(t)}\right)^{\beta} \phi'(t)\\
&=K \left(\E^{1/\log(1+\E^t)}\right)^{\beta}
\left(\frac{\E^{-t}}{1+\E^{-t}}\right)^{\beta}
\frac{1+\{\log(1+\E^t)\}^2}{(1+\E^{-t})\{\log(1+\E^t)\}^2}\\
&\leq K
\E^{\beta/\log 2}
\frac{\E^{-\beta t}}{(1+0)^{\beta+1}}
\left\{1+\frac{1}{(\log 2)^2}\right\}.
\end{align*}
Using this estimate, we have
\begin{align*}
\left|h\sum_{k=N+1}^{\infty}F(kh)\right|
&\leq h \sum_{k=N+1}^{\infty} |F(kh)|\\
&\leq K \E^{\beta/\log 2} \left\{1+\frac{1}{(\log 2)^2}\right\}
h\sum_{k=N+1}^{\infty}\E^{-\beta k h}\\
%&\leq K \E^{\beta\delta} \left(1+\delta^2\right)
%h\sum_{k=N+1}^{\infty}\frac{1}{(0+\E^{kh})^{\beta}}\\
&\leq K \E^{\beta/\log 2} \left\{1+\frac{1}{(\log 2)^2}\right\}
\int_{Nh}^{\infty}\E^{-\beta x}\D x\\
&= K \E^{\beta/\log 2} \left\{1+\frac{1}{(\log 2)^2}\right\}
   \frac{\E^{-\beta Nh}}{\beta}.
\end{align*}
Next, from~\eqref{leq:f-Dd-minus-original},
using~\eqref{ineq:exp-log-real}, \eqref{leq:func-bound-Ddminus},
and~\eqref{leq:real-func-bound},
it holds for all $t\leq 0$ that
\begin{align*}
|F(t)|
&\leq K \frac{1}{|\phi(t)|^{\alpha+1}}\phi'(t)\\
&= K \left|\frac{\log(1+\E^t)}{1+\log(1+\E^t)}\right|^{\alpha+1}
\frac{1 + \E^{t}}{\E^t|-1 + \log(1+\E^t)|^{\alpha+1}}
\frac{1+\{\log(1+\E^t)\}^2}{(1+\E^{-t})^2\{\log(1+\E^{t})\}^2}\\
&\leq K\left(\frac{\E^t}{1+\E^t}\right)^{\alpha+1}
 \frac{1+\E^{t}}{\E^t(1 - \log 2)^{\alpha+1}}
 \E^{1/\pi^3}\\
&\leq K \frac{\E^{\alpha t}}{(1+ 0)^{\alpha}}
 \frac{1}{(1 - \log 2)^{\alpha+1}} \E^{1/\pi^3}.
\end{align*}
Using this estimate, we have
\begin{align*}
\left|h\sum_{k=-\infty}^{-M-1}F(kh)\right|
&\leq h \sum_{k=-\infty}^{-M-1} |F(kh)|\\
&\leq K \frac{\E^{1/\pi^3}}{(1 - \log 2)^{\alpha+1}}
h\sum_{k=-\infty}^{-M-1}\E^{\alpha k h}\\
&\leq K \frac{\E^{1/\pi^3}}{(1 - \log 2)^{\alpha+1}}
\int_{-\infty}^{-Mh}\E^{\alpha x}\D x\\
&= K\frac{\E^{1/\pi^3}}{(1 - \log 2)^{\alpha+1}}
   \frac{\E^{-\alpha Mh}}{\alpha}.
\end{align*}
Thus, using~\eqref{eq:Def-MN},
we have
\begin{align*}
\left|h\sum_{k=-\infty}^{-M-1}F(kh)+h\sum_{k=N+1}^{\infty}F(kh)\right|
&\leq
\frac{K\E^{1/\pi^3}}{\alpha(1 - \log 2)^{\alpha+1}}
   \E^{-\alpha Mh}
+\frac{K\E^{\beta/\log 2}}{\beta} \left\{1+\frac{1}{(\log 2)^2}\right\}
  \E^{-\beta Nh}\\
&\leq
\frac{K\E^{1/\pi^3}}{\alpha(1 - \log 2)^{\alpha+1}}
   \E^{-\mu n h}
+\frac{K\E^{\beta/\log 2}}{\beta} \left\{1+\frac{1}{(\log 2)^2}\right\}
  \E^{-\mu n h},
\end{align*}
which is the desired estimate.
\end{proof}

\section{Proofs for Theorem~\ref{thm:New2}}
\label{sec:proofs-2}

This section presents the proof of Theorem~\ref{thm:New2}.
It is organized as follows.
In Section~\ref{subsec:sketch-proof-2},
the task is decomposed into two lemmas:
Lemmas~\ref{lem:bound-None-2} and~\ref{lem:bound-truncation-error-2}.
To prove these lemmas,
a useful inequality is presented in
Section~\ref{subsec:domDd-2}.
Then,
Lemma~\ref{lem:bound-None-2} is proved in
Section~\ref{subsec:discretization-error-2},
and Lemma~\ref{lem:bound-truncation-error-2} is proved in
Section~\ref{subsec:truncation-error-2}.

\subsection{Sketch of the proof}
\label{subsec:sketch-proof-2}

%Let $F(t)=f(\phi(t))\phi'(t)$.
The main strategy in the proof of Theorem~\ref{thm:New2} is
identical to that of Theorem~\ref{thm:New1}, that is,
splitting the error into the discretization error
and the truncation error
as~\eqref{eq:decompose-disc-trun}.
For the discretization error,
we show the following lemma.
The proof is given in Section~\ref{subsec:discretization-error-2}.

\begin{lemma}
%{Lemma 4.3}
\label{lem:bound-None-2}
Let the assumptions made in Theorem~\ref{thm:New2}
be fulfilled.
Then, the function $F(\zeta)=f(\phi(\zeta))\phi'(\zeta)$
belongs to $\mathbf{H}^1(\domD_d)$, and $\mathcal{N}_1(F,d)$ is bounded as
\[
 \mathcal{N}_1(F,d)
\leq
2 K C_5,
\]
where $C_5$ is a constant defined as~\eqref{def:C_5}.
\end{lemma}

In addition, we bound the truncation error as follows.
The proof is given in Section~\ref{subsec:truncation-error-2}.

\begin{lemma}
%{Lemma 4.4}
\label{lem:bound-truncation-error-2}
Let the assumptions made in Theorem~\ref{thm:New2}
be fulfilled.
Then, setting $F(\zeta)=f(\phi(\zeta))\phi'(\zeta)$, we have
\begin{align*}
\left|
h\sum_{k=-\infty}^{-M-1} F(kh)
+
h\sum_{k=N+1}^{\infty} F(kh)
\right|
\leq K C_6 \E^{-\mu n h},
\end{align*}
where $C_6$ is a constant defined as~\eqref{def:C_6}.
\end{lemma}

Setting $h$ as~\eqref{eq:Def-h}, the above estimates
(Theorem~\ref{thm:bound-discretization-error},
Lemmas~\ref{lem:bound-None-2}, and~\ref{lem:bound-truncation-error-2})
yield the desired result as
\begin{align*}
\left|
\int_{-\infty}^{\infty}f(x)\D x
- h\sum_{k=-M}^N f(\phi(kh))\phi'(kh)
\right|
&\leq \frac{2KC_5}{1-\E^{-2\pi d/h}}\E^{-2\pi d/h}
+ K C_6 \E^{-\mu n h}\\
&= K\left( \frac{2 C_5}{1-\E^{-\sqrt{2\pi d \mu n}}}
+ C_6\right)\E^{-\sqrt{2\pi d \mu n}}\\
&\leq K\left( \frac{2 C_5}{1-\E^{-\sqrt{2\pi d \mu}}}
+ C_6\right)\E^{-\sqrt{2\pi d \mu n}}.
\end{align*}
This completes the proof of Theorem~\ref{thm:New2}.

\subsection{Useful inequality on $\domD_d$}
\label{subsec:domDd-2}

We prepare the following lemma here.

\begin{lemma}
%{Lemma 4.10}
Let $d$ be a constant satisfying $0<d<\pi$.
For all $\zeta\in \domD_{d}$ and $x\in\mathbb{R}$, we have
\begin{align}
\left|
\frac{1}{(1+\E^{-\zeta})\log(1+\E^{\zeta})}
\right|
&\leq \frac{1+c_d}{\log(2+c_d)},
 \label{leq:Dd-func-bound-2} \\
\frac{1}{(1+\E^{-x})\log(1+\E^x)}
&\leq 1,
\label{leq:real-func-bound-2}
\end{align}
where $c_d=1/\cos(d/2)$.
\end{lemma}
\begin{proof}
First, consider~\eqref{leq:real-func-bound-2}, which is proved by showing
\[
 p(t) = \frac{1-\E^{-t}}{t} \leq 1
%\label{leq:real-func-bound-shown-2}
\]
for all $t > 0$ (put $t=\log(1+\E^x)$).
Differentiating $p(x)$, we have
\[
 p'(t) = -\frac{\E^t - (1+t)}{\E^t t^2} \leq 0,
\]
since $\E^t \geq 1 + t$ holds.
Therefore, $p(t)$ decreases monotonically,
and thus, it holds that $p(t)\leq \lim_{t\to 0}p(t) = 1$.

Next, we show~\eqref{leq:Dd-func-bound-2}. Let
$g(\zeta)=1/\{(1+\E^{-\zeta})\log(1+\E^{\zeta})\}$.
Since the function $g(\zeta)$
is analytic and bounded in $\domD_d$ and continuous
on $\overline{\domD_d}$,
by Lemma~\ref{lem:three-lines},
we obtain~\eqref{leq:Dd-func-bound-2} if we show the following two inequalities:
\begin{align*}
M_g(d) \leq \frac{1+c_d}{\log(2+c_d)},
\quad
M_g(-d) \leq \frac{1+c_d}{\log(2+c_d)},
\end{align*}
where $M_g(y) = \sup_{x\in\mathbb{R}}|g(x+\I y)|$.
We show only the first one, because the second one is also shown
in the same way.
Putting $\xi=\log(1+\E^{x + \I d})$, $g(x + \I d)=p(\xi)$ holds,
%where the function $p$ is introduced in~\eqref{leq:real-func-bound-shown}.
and thus, in what follows we prove
\begin{equation}
|p(\xi)|\leq \frac{1+c_d}{\log(2+c_d)}.
\label{leq:Dd-func-bound-target-2}
\end{equation}
We consider the following two cases: (a) $|\xi|\leq \log(2+c_d)$
and (b) $|\xi|>\log(2+c_d)$.
In case (a), we have
\begin{align*}
|p(\xi)|
=\left|\sum_{k=1}^{\infty}\frac{(-\xi)^{k-1}}{k!}\right|
\leq \sum_{k=1}^{\infty}\frac{|\xi|^{k-1}}{k!}
=\frac{\E^{|\xi|}-1}{|\xi|}.
\end{align*}
Here, if we put $q(x)=(\E^x - 1)/x$,
then we have $q'(x)=r(x)/x^2$, where $r(x)=1 + (x-1)\E^x$.
Since $r'(x)=x\E^x \geq 0$ for $x\geq 0$,
$r(x)$ monotonically increases for $x\geq 0$.
Therefore, $r(x)\geq r(0)=0$ holds, from which we have
$q'(x)\geq 0$ for $x\geq 0$, i.e., $q(x)$ monotonically increases
for $x\geq 0$.
Thus, from $|\xi|\leq \log(2+c_d)$, we have~\eqref{leq:Dd-func-bound-target-2}
as
\[
 |p(\xi)|\leq q(|\xi|)
 \leq \frac{\E^{\log(2+c_d)} - 1}{\log(2+c_d)}
=\frac{1+c_d}{\log(2+c_d)}.
\]
In case (b), using~\eqref{eq:bound-re-log},
we have
\[
 |p(\xi)|\leq \frac{1 + |\E^{-\xi}|}{|\xi|}
=\frac{1 + \E^{-\Re\xi}}{|\xi|}
\leq \frac{1 + \E^{-\log(\cos(d/2))}}{|\xi|}
= \frac{1 + c_d}{|\xi|}.
\]
Furthermore, since $1/x$ decreases monotonically for $x>0$,
we have~\eqref{leq:Dd-func-bound-target-2}.
This completes the proof.
\end{proof}

\subsection{Estimation of the discretization error (proof of Lemma~\ref{lem:bound-None-2})}
\label{subsec:discretization-error-2}

Lemma~\ref{lem:bound-None-2} is essentially shown by the following lemma,
which holds for $0<\delta<\pi$ (not only $0<d<(1+\pi)/2$).

\begin{lemma}
\label{lemma:essential-for-None-2}
Assume that $F$ is analytic in $\domD_{\delta}$ with $0<\delta<\pi$,
and that there exist positive constants $K_{+}$, $K_{-}$,
$\alpha$, and $\beta$
such that
\begin{align}
 |F(\zeta)|&\leq K_{+}\left|
\frac{\E^{1/\log(1+\E^{\zeta})}}{1+\E^{\zeta}}
\right|^{\beta}\label{leq:f-Dd-plus-strip}
\intertext{holds for all $\zeta\in\domD_{\delta}^{+}$, and}
|F(\zeta)|&\leq K_{-}\left|
\frac{\log(1+\E^{\zeta})}{\{1+\log(1+\E^{\zeta})\}\{-1+\log(1+\E^{\zeta})\}}
\right|^{\alpha}\label{leq:f-Dd-minus-strip}
\end{align}
holds for all $\zeta\in\domD_{\delta}^{-}$.
Then, $F$ belongs to $\mathbf{H}^1(\domD_{\delta})$,
and $\mathcal{N}_1(F,\delta)$ is bounded as
\begin{equation}
 \mathcal{N}_1(F,\delta)\leq
\frac{2K_{-}}{\alpha}\left\{\frac{\E c_{\delta}}{(1-\log 2)(\E-1)}\right\}^{\alpha}
+\frac{2K_{+}}{\beta}\left(\E^{1/\log 2} c_{\delta}\right)^{\beta},
\label{eq:bound-None-delta}
\end{equation}
where $c_{\delta}=1/\cos(\delta/2)$.
\end{lemma}
\begin{proof}
Since $F$ is analytic on $\domD_{\delta}$, the remaining task is
to show~\eqref{eq:bound-None-delta}.
From~\eqref{leq:f-Dd-plus-strip},
by using~\eqref{leq:Lemma45},
it holds for all $\zeta\in\domD_{\delta}^{+}$ that
\begin{align}
|F(\zeta)|&\leq K_{+}
\left|\E^{1/\log(1+\E^{\zeta})}\right|^{\beta}
\left|\frac{1}{1+\E^{\zeta}}\right|^{\beta}
\nonumber\\
%&\leq K
%\left|\E^{1/\arcsinh(1)}\right|^{\beta}
%\left|\frac{\sqrt{2}}{1+\E^{\zeta}}\right|^{\beta}
%\frac{1}{|1+\E^{-2\zeta}|^{1/2}}
%\left(1+\frac{1}{\arcsinh^2(1)}\right)\\
&\leq K_{+} \left(\E^{1/\log 2}\right)^{\beta}
\frac{1}{|1+\E^{\zeta}|^{\beta}}.
\label{leq:bound-F-Dd-plus-strip}
\end{align}
Furthermore, from~\eqref{leq:f-Dd-minus-strip},
by using~\eqref{leq:func-bound-Ddminus} and~\eqref{ineq:exp-log-complex},
it holds for all $\zeta\in\domD_{\delta}^{-}$ that
\begin{align}
|F(\zeta)|
&\leq
K_{-}\left|\frac{\log(1+\E^{\zeta})}{1+\log(1+\E^{\zeta})}\right|^{\alpha}
\frac{1}{|-1+\log(1+\E^{\zeta})|^{\alpha}}
\nonumber \\
&\leq K_{-} \frac{1}{|1+\E^{-\zeta-l}|^{\alpha}}
\frac{1}{(1-\log 2)^{\alpha}},
\label{leq:bound-F-Dd-minus-strip}
\end{align}
where $l=\log(\E/(\E - 1))$.
As described earlier, $\mathcal{N}_1(F,d)$ is expressed as~\eqref{eq:def-None}.
Using~\eqref{ineq:exp-1-plus}
and~\eqref{leq:bound-F-Dd-plus-strip},
we have
\begin{align*}
\int_{-\delta(1-\epsilon)}^{\delta(1-\epsilon)}|F(1/\epsilon + \I y)|\D y
&\leq K_{+} \left(\E^{1/\log 2}\right)^{\beta}
\int_{-\delta(1-\epsilon)}^{\delta(1-\epsilon)}
\frac{1}{|1+\E^{1/\epsilon+\I y}|^{\beta}}\D y\\
&\leq \frac{K_{+} \left(\E^{1/\log 2}\right)^{\beta}}
           {(1+\E^{1/\epsilon})^{\beta}}
\int_{-\delta(1-\epsilon)}^{\delta(1-\epsilon)}
\frac{1}{\cos^{\beta}(y/2)}\D y,
\end{align*}
from which we have
\[
 \lim_{\epsilon\to 0}
\int_{-\delta(1-\epsilon)}^{\delta(1-\epsilon)}|F(1/\epsilon + \I y)|\D y
= 0.
\]
In the same manner,
%with regard to the fourth term of~\eqref{eq:def-None},
using~\eqref{ineq:exp-1-minus}
and~\eqref{leq:bound-F-Dd-minus-strip},
we have
\[
 \lim_{\epsilon\to 0}
\int_{-\delta(1-\epsilon)}^{\delta(1-\epsilon)}|F(-1/\epsilon + \I y)|\D y
= 0.
\]
Therefore, $\mathcal{N}_1(F,\delta)$ is expressed as
\begin{align*}
\mathcal{N}_1(F,\delta)
&=\int_{-\infty}^{\infty}|F(x-\I \delta)|\D x
 +\int_{-\infty}^{\infty}|F(x+\I \delta)|\D x\\
&=\int_{-\infty}^{0}|F(x-\I \delta)| \D x
 +\int_{0}^{\infty}|F(x-\I \delta)|\D x
 +\int_{-\infty}^{0}|F(x+\I \delta)|\D x
 +\int_{0}^{\infty}|F(x+\I \delta)|\D x.
%\label{eq:None-rewrite}
\end{align*}
With regard to the first term,
using~\eqref{ineq:exp-1-minus}
and~\eqref{leq:bound-F-Dd-minus-strip},
we have
\begin{align*}
\int_{-\infty}^{0}|F(x-\I \delta)|\D x
&\leq \frac{K_{-}}{(1- \log 2)^{\alpha}}
\int_{-\infty}^{0}
\frac{1}{|1+\E^{-x-l+\I \delta}|^{\alpha}}
\D x\\
&\leq \frac{K_{-}}{(1- \log 2)^{\alpha}}
\int_{-\infty}^{0}
\frac{1}{(1+\E^{-x-l})^{\alpha}\cos^{\alpha}(\delta/2)}
\D x\\
&=\frac{K_{-} c_{\delta}^{\alpha}}{(1- \log 2)^{\alpha}}
\int_{-\infty}^{0}
\frac{1}{(1+\E^{-x-l})^{\alpha}}
\D x.
\end{align*}
The integral is further bounded as
\begin{align*}
\int_{-\infty}^{0}
\frac{1}{(1+\E^{-x-l})^{\alpha}}
\D x
&=\int_{-\infty}^{0}
\frac{\E^{\alpha x}}{(\E^x+\E^{-l})^{\alpha}}
\D x\\
&\leq
\int_{-\infty}^{0}
\frac{\E^{\alpha x}}{(0+\E^{-l})^{\alpha}}
\D x\\
&=\left(\frac{\E}{\E - 1}\right)^{\alpha}\frac{1}{\alpha}.
\end{align*}
In the same manner, the third term is bounded as
\[
\int_{-\infty}^{0}|F(x+\I \delta)|\D x
\leq
\frac{K_{-} c_{\delta}^{\alpha}}{\alpha(1- \log 2)^{\alpha}}
\left(\frac{\E}{\E - 1}\right)^{\alpha}.
\]
With regard to the second term,
using~\eqref{ineq:exp-1-plus}
and~\eqref{leq:bound-F-Dd-plus-strip},
we have
\begin{align*}
\int_{0}^{\infty}|F(x-\I \delta)|\D x
&\leq K_{+} \E^{\beta/\log 2}
\int_{0}^{\infty}
\frac{1}{|1+\E^{x-\I \delta}|^{\beta}}
\D x\\
&\leq  K_{+} \E^{\beta/\log 2}
\int_{0}^{\infty}
\frac{1}{(1+\E^x)^{\beta}\cos^{\beta}(\delta/2)}
\D x\\
&=K_{+} \E^{\beta/\log 2}c_{\delta}^{\beta}
\int_{0}^{\infty}
\frac{\E^{-\beta x}}{(1+\E^{-x})^{\beta}}
\D x\\
&\leq K_{+} \E^{\beta/\log 2} c_{\delta}^{\beta}
\int_{0}^{\infty}
\frac{\E^{-\beta x}}{(1+0)^{\beta}}
\D x\\
&=K_{+} \E^{\beta/\log 2}
\frac{c_{\delta}^{\beta}}{\beta}.
\end{align*}
In the same manner, the fourth term is bounded as
\[
\int_{0}^{\infty}|F(x+\I \delta)|\D x
\leq K_{+} \E^{\beta/\log 2}
\frac{c_{\delta}^{\beta}}{\beta}.
\]
Thus, we obtain~\eqref{eq:bound-None-delta}.
\end{proof}

Using this lemma,
Lemma~\ref{lem:bound-None-2} is shown as follows.

\begin{proof}
Let $F(\zeta)=f(\phi(\zeta))\phi'(\zeta)$.
Since $f$ is analytic in $\phi(\domD_d)$,
$f(\phi(\cdot))$ is analytic in $\domD_d$.
In addition, since $\phi'$ is analytic in $\domD_{\pi}$,
$F$ is analytic in $\domD_d$ (note that $d<\pi$).
Therefore, the remaining task is to show
$\mathcal{N}_1(F,d)\leq 2KC_5$.
Using~\eqref{ineq:exp-1-minus}, we have
\[
\frac{1}{|1+\E^{-\zeta}|}
\leq \frac{1}{(1+\E^{-\Re\zeta})\cos((\Im\zeta)/2)}
\leq \frac{1}{(1+ 0)\cos(d/2)}
\]
for all $\zeta\in\domD_d$.
Therefore, from~\eqref{leq:f-Dd-plus},
by using~\eqref{leq:Lemma46},
it holds for all $\zeta\in\domD_d^{+}$ that
\begin{align*}
 |F(\zeta)|&\leq K|\E^{-\phi(\zeta)}|^{\beta}|\phi'(\zeta)| \\
&=K\left|\E^{1/\log(1+\E^{\zeta})}\right|^{\beta}
\left|\frac{1}{1+\E^{\zeta}}\right|^{\beta}
\frac{|1+\{\log(1+\E^{\zeta})\}^2|}{|1+\E^{-\zeta}||\log(1+\E^{\zeta})|^2}\\
&\leq K\left|\frac{\E^{1/\log(1+\E^{\zeta})}}{1+\E^{\zeta}}\right|^{\beta}
c_d \left\{1+ \frac{1}{(\log 2)^2}\right\},
\end{align*}
where $c_d=1/\cos(d/2)$.
Furthermore, from~\eqref{leq:f-Dd-minus-new},
by using~\eqref{leq:Dd-func-bound-2},
it holds for all $\zeta\in\domD_d^{-}$ that
\begin{align*}
|F(\zeta)|
&\leq K \frac{1}{|4 + \{\phi(\zeta)\}^2|^{1/2}|\phi(\zeta)|^{\alpha}}
\left|\phi'(\zeta)\right|\\
&= K \left|\frac{\log(1+\E^{\zeta})}{1+\log(1+\E^{\zeta})}\right|^{\alpha}
\frac{1}{|-1+\log(1+\E^{\zeta})|^{\alpha}}
\frac{1}{|(1 + \E^{-\zeta})\log(1+\E^{\zeta})|}\\
&\leq K \left|
\frac{\log(1+\E^{\zeta})}{\{1+\log(1+\E^{\zeta})\}\{-1+\log(1+\E^{\zeta})\}}
\right|^{\alpha}
\frac{1+c_d}{\log(2+c_d)}.
\end{align*}
Thus, the assumptions of Lemma~\ref{lemma:essential-for-None-2}
are fulfilled with $\delta=d$ and
\begin{align*}
K_{+} &= K c_d\left\{1 + \frac{1}{(\log 2)^2}\right\},\\
K_{-} &= K \frac{1 + c_d}{\log(2 + c_d)},
\end{align*}
from which we have $\mathcal{N}_1(F,d)\leq 2KC_5$.
\end{proof}

\subsection{Estimation of the truncation error (proof of Lemma~\ref{lem:bound-truncation-error-2})}
\label{subsec:truncation-error-2}

Lemma~\ref{lem:bound-truncation-error-2}
is essentially shown by the following lemma.

\begin{lemma}
\label{lemma:essential-for-truncate-2}
Assume that there exist positive constants
$K_{+}$, $K_{-}$, $\alpha$, and $\beta$ such that
\begin{align}
 |F(x)|&\leq K_{+}\left|
\frac{\E^{1/\log(1+\E^{x})}}{1+\E^{x}}
\right|^{\beta}\label{leq:f-Dd-plus-real}
\intertext{holds for all $x\geq 0$, and}
|F(x)|&\leq K_{-}\left|
\frac{\log(1+\E^{x})}{\{1+\log(1+\E^{x})\}\{-1+\log(1+\E^{x})\}}
\right|^{\alpha}\label{leq:f-Dd-minus-real}
\end{align}
holds for all $x < 0$.
Let $\mu = \min\{\alpha,\beta\}$, and
let $M$ and $N$ be defined as~\eqref{eq:Def-MN}.
Then, we have
\begin{equation}
h\sum_{k=-\infty}^{-M-1}|F(kh)|
+ h\sum_{k=N+1}^{\infty}|F(kh)|
\leq \left\{
\frac{K_{-}}{\alpha(1 - \log 2)^{\alpha}}
+\frac{K_{+}}{\beta}\left(\E^{1/\log 2}\right)^{\beta}
\right\}
\E^{-\mu n h}.
\label{eq:bound-truncate-2}
\end{equation}
\end{lemma}
\begin{proof}
From~\eqref{leq:f-Dd-plus-real},
by using~\eqref{leq:Lemma45},
%and~\eqref{ineq:exp-asinh-real},
it holds for all $x\geq 0$ that
\begin{align*}
|F(x)|\leq
K_{+} \left(\E^{1/\log(1+\E^x)}\right)^{\beta}
\left(\frac{\E^{-x}}{1+\E^{-x}}\right)^{\beta}
\leq K_{+}
\E^{\beta/\log 2}
\frac{\E^{-\beta x}}{(1+0)^{\beta}}.
\end{align*}
Using this estimate, we have
\begin{align*}
h\sum_{k=N+1}^{\infty}|F(kh)|
\leq K_{+} \E^{\beta/\log 2}
h\sum_{k=N+1}^{\infty}\E^{-\beta k h}
%&\leq K \E^{\beta\delta} \left(1+\delta^2\right)
%h\sum_{k=N+1}^{\infty}\frac{1}{(0+\E^{kh})^{\beta}}\\
\leq K_{+} \E^{\beta/\log 2}
\int_{Nh}^{\infty}\E^{-\beta x}\D x
= K_{+} \E^{\beta/\log 2}
   \frac{\E^{-\beta Nh}}{\beta}.
\end{align*}
Next, from~\eqref{leq:f-Dd-minus-real},
using~\eqref{ineq:exp-log-real} and~\eqref{leq:func-bound-Ddminus},
it holds for all $x < 0$ that
\begin{align*}
|F(x)|
\leq K_{-} \left|\frac{\log(1+\E^x)}{1+\log(1+\E^x)}\right|^{\alpha}
\frac{1}{|-1 + \log(1+\E^x)|^{\alpha}}
\leq K_{-}\left(\frac{\E^x}{1+\E^x}\right)^{\alpha}
 \frac{1}{(1 - \log 2)^{\alpha}}
\leq K_{-} \frac{\E^{\alpha x}}{(1+ 0)^{\alpha}}
 \frac{1}{(1 - \log 2)^{\alpha}}.
\end{align*}
Using this estimate, we have
\begin{align*}
h \sum_{k=-\infty}^{-M-1} |F(kh)|
\leq  \frac{K_{-}}{(1 - \log 2)^{\alpha}}
h\sum_{k=-\infty}^{-M-1}\E^{\alpha k h}
\leq \frac{K_{-}}{(1 - \log 2)^{\alpha}}
\int_{-\infty}^{-Mh}\E^{\alpha x}\D x
= \frac{K_{-}}{(1 - \log 2)^{\alpha}}
   \frac{\E^{-\alpha Mh}}{\alpha}.
\end{align*}
Thus, using~\eqref{eq:Def-MN}, we have~\eqref{eq:bound-truncate-2}.
\end{proof}

Using this lemma,
Lemma~\ref{lem:bound-truncation-error-2} is shown as follows.

\begin{proof}
Let $F(x)=f(\phi(x))\phi'(x)$.
From~\eqref{leq:f-Dd-plus},
by using~\eqref{leq:Lemma46},
it holds for all $x\geq 0$ that
\begin{align*}
 |F(x)|&\leq K|\E^{-\phi(x)}|^{\beta}|\phi'(x)| \\
&=K\left|\E^{1/\log(1+\E^{x})}\right|^{\beta}
\left|\frac{1}{1+\E^{x}}\right|^{\beta}
\frac{1}{1+\E^{-x}}
\frac{1+\{\log(1+\E^{x})\}^2}{\{\log(1+\E^{x})\}^2}\\
&\leq K\left|\frac{\E^{1/\log(1+\E^{x})}}{1+\E^{x}}\right|^{\beta}
\frac{1}{1 + 0}
\left\{1+ \frac{1}{(\log 2)^2}\right\}.
\end{align*}
Next, from~\eqref{leq:f-Dd-minus-new},
using~\eqref{leq:real-func-bound-2},
it holds for all $x< 0$ that
\begin{align*}
|F(x)|
&\leq K \frac{1}{|1 + \{\phi(x)\}^2|^{1/2}|\phi(x)|^{\alpha}}\phi'(x)\\
&= K \left|\frac{\log(1+\E^x)}{1+\log(1+\E^x)}\right|^{\alpha}
\frac{1}{|-1 + \log(1+\E^x)|^{\alpha}}\cdot
\frac{1}{(1+\E^{-x})\log(1+\E^{x})}\\
&\leq K\left|
\frac{\log(1+\E^x)}{\{1+\log(1+\E^x)\}\{-1+\log(1+\E^x)\}}
\right|^{\alpha}\cdot 1.
\end{align*}
Thus,
the assumptions of Lemma~\ref{lemma:essential-for-truncate-2}
are fulfilled with
\begin{align*}
K_{+} &= K\left(1 + \frac{1}{(\log 2)^2}\right),\\
K_{-} &= K,
\end{align*}
which completes the proof.
\end{proof}

%\section*{Acknowledgments}
%This work was partially supported by the
%JSPS Grant-in-Aid for Young Scientists (B)
%Number JP17K14147.

\bibliography{NewConformalMapUnilateral}

\end{document}